\documentclass[12pt]{article}
\usepackage[leqno]{amsmath}
\usepackage{amssymb, amsthm}
\usepackage{graphicx}
\usepackage{color}
\newtheorem{thm}{Theorem}
\newtheorem{cor}[thm]{Corollary}

\newtheorem{lem}[thm]{Lemma}
\newtheorem{prop}[thm]{Proposition}

\newtheorem{exam}[thm]{Example}

\def\C{\mathbb{C}}

\def\N{\mathbb{N}}

\def\R{\mathbb{R}}
\def\S{\mathbb{S}}

\title{Pervasive Algebras and Maximal Subalgebras}

\author{Pamela Gorkin and Anthony G. O'Farrell}

\begin{document}
\maketitle

\begin{abstract}
A uniform algebra $A$ on its Shilov boundary $X$ is {\em maximal}
if $A$ is not $C(X)$ 
and there is no uniform algebra properly contained between
$A$ and $C(X)$. It is {\em essentially pervasive} if
$A$ is dense in $C(F)$ whenever $F$ is a proper closed
subset of the essential set of $A$.  If $A$ is maximal,
then it is essentially 
pervasive and proper. 
We explore the gap between
these two concepts. We show the following: (1) If
$A$ is pervasive and proper, and 
has a nonconstant unimodular element, then $A$
contains an infinite descending chain of pervasive subalgebras
on $X$. (2) It is possible to imbed a copy of 
the lattice of all subsets
of $\N$ into the family of pervasive subalgebras of some $C(X)$.
(3) In the other direction, if $A$ is strongly logmodular,
proper and pervasive, then it is maximal. (4) This fails
if the word \lq strongly' is removed. 

We discuss further examples, involving Dirichlet
algebras, $A(U)$ algebras, Douglas algebras,
and subalgebras of $H^\infty(\mathbb{D})$. We develop some new results that 
relate pervasiveness, maximality and relative
maximality to support sets
of representing measures.  
\end{abstract}


\section{Introduction}

This paper is about pervasive uniform algebras, and the connections between
pervasiveness, maximality, and the presence of nonconstant 
unimodular functions.

For a compact Hausdorff space $X$, let $C(X)=C(X,\C)$ denote the 
algebra of complex-valued continuous functions on $X$, equipped with the
topology induced by the uniform norm.  
Concretely, a {\em uniform algebra on $X$}
is a closed subalgebra $A$ of $C(X)$ that  
contains the constants and separates the points of $X$. 
Each uniform algebra is an example of a commutative, semisimple
Banach algebra, i.e a complete normed complex algebra \cite{Ga}.
Each semisimple commutative Banach algebra $A$ may be regarded,
via the Gelfand transform, as an algebra of complex-valued functions
on its maximal ideal space, or character space, $M(A)$.
(The Gelfand transform is defined by
$$ \hat f(\phi) = \phi(f),\ \forall \phi\in M(A), \forall f\in A).$$
Abstractly, the uniform algebras are characterised among
semisimple commutative Banach algebras by the property
that they are complete with respect to the uniform norm
on $M(A)$. For this, and other general facts about uniform
algebras referred to below, see \cite{Ga}.

If $A$ is a uniform algebra on $X$, then $X$ is homeomorphic 
to a closed subset of $M(A)$, when $M(A)$ is given the
weak-star topology inherited from the dual $A^*$ (-- characters
belong to $A^*$).  It is customary to identify $X$ with its
image in $M(A)$. When we do this, it always happens that $X$
includes the {\em Shilov boundary} of $A$, the minimal
closed subset $Y$ of $M(A)$ such that
$$\|f\| \le \sup\{|\psi(f)|: \psi\in Y\},\ \forall f\in A.$$
In general, a uniform algebra $A$ is isometrically isomorphic to the restriction
algebra $A|S$, where $S$ is the Shilov boundary of $A$, and
one normally identifies $A$ with $A|S$.

\medskip
Here are the definitions of the main concepts we study:

\medskip 
$A$ is a {\it maximal subalgebra} of $C(X)$ if $A$ is properly 
contained in $C(X)$ (that is, $A \subset C(X)$) and there is no 
closed algebra $B$ satisfying $A \subset B \subset C(X)$.  We also
express this by saying that $A$ is {\em maximal on $X$}.

\medskip 
Let $A$ be a uniform algebra on $X\subseteq M(A)$. Then $A$ 
is said to be {\em pervasive on $X$}, if for every proper 
compact subset $Y$ of $X$, the restriction algebra 
$$A|Y = \{f|Y: f\in A\}$$ 
is dense in $C(Y)$.  The algebra $A$ is said to be {\em pervasive}
if it is pervasive on its Shilov boundary.  
This concept was first introduced
by Hoffman and Singer \cite{HS}, in a paper devoted to aspects
of maximality.  
It is not hard to see that if $A$ is pervasive on $X$, then
$X$ is in fact the Shilov boundary of $A$.

\medskip
The archetypical example is the {\em disk algebra}, 
consisting of all the functions continuous on the closed unit disk, 
$\textup{clos}(\mathbb{D})$, and analytic on its interior. Its
maximal ideal space is $\textup{clos}(\mathbb{D})$, and
its Shilov boundary is  
the unit circle $\mathbb{S}^1$. 
This 
algebra is pervasive. This amounts to saying that the 
analytic polynomials are
dense in all continuous functions on any proper subarc of the unit
circle. 
This is a special case of Lavrentieff's Theorem \cite[Theorem II.8.7]{Ga},
but may be proved in many ways. 

According to the celebrated 
Wermer Maximality Theorem \cite[Theorem II.5.1]{Ga},
the disk algebra is also maximal, and this suggests that
there might be some connection between being maximal
and being pervasive. 
Hoffman and Singer introduced pervasive algebras, 
motivated by the observation that 
some results of Helson and Quigley, {\it inter alia},
about other uniform algebras that 
behaved like the disk algebra, were, logically, more directly
connected with its pervasiveness than its maximality.
They established some general 
connections between maximality and pervasiveness.

Strictly speaking, maximality and pervasiveness are quite
distinct properties, logically unconnected in the sense
that the truth or falsehood of one tells you nothing about
the other (see Section \ref{basicexamples} below).  
However, one can say that \lq\lq essentially",
maximality implies pervasiveness. More precisely,
if $A$ is maximal, then the restriction of $A$
to its so-called essential set\footnote{For this, and 
other terms used in this introduction and not yet
defined, see the next section.} is pervasive there.
If we confine attention to essential algebras
(those with Shilov boundary equal to the essential set),
then this raises the question of what must be
added to the assumption of
pervasiveness to ensure maximality.

We give examples to show that the gap from pervasiveness
to maximality may be very large. In the most
striking example (cf. Subsection \ref{longcomplicated}), 
the family of pervasive subalgebras
of $C(X)$ contains an isomorphic copy of the lattice
of all subsets of $\N$, the set of natural numbers.

One may ask whether it makes any difference if the algebra
is assumed to be Dirichlet. In fact, in the example
referred to, all the pervasive algebras we construct
are Dirichlet. However, we do show (Section \ref{maximal})
that if the algebra $A$ is assumed strongly logmodular, if $A$ is not maximal then $A$ 
is not pervasive. Since 
Dirichlet algebras are logmodular, the example 
of Subsection \ref{longcomplicated} shows that 
logmodular algebras can be essential, pervasive 
and nonmaximal.

For all connected open subsets $U$ of the Riemann sphere  
$\hat{\mathbb{C}}$, Hoffman and Singer considered
the associated algebra $A(U)$, consisting of those continuous
functions on $\textup{clos}(U)$ that are holomorphic on $U$,
regarded as a uniform algebra on its Shilov boundary, $X$.
They showed that in some cases $A(U)$ is pervasive on $X$.
Gamelin and Rossi \cite{GR} showed that $A(U)$ is in fact
maximal in $C(X)$ whenever $U$ is connected, $U$ is the
interior of its closure $K$, $K\subset\C$ is 
compact, and each $f\in A(U)$ may be approximated
uniformly on $K$ by rational functions with poles off $X$. 
They raised the question of the maximality of more general
$A(U)$, but to date this has not been completely 
resolved, except in the cases when $A(U)$ is Dirichlet
\cite[p. 63, Ex. 1]{Ga}.  See also
\cite{Davie} and subsection \ref{SS:Warc}.
Later, one of the authors and collaborators showed that,
for all connected open subsets $U$ of the Riemann sphere  
$\hat{\mathbb{C}}$,
the associated algebra $A(U)$
is pervasive \cite[Theorem 3.2]{NOS}. Further work
extended this to open Riemann surfaces \cite{OS}.
There is also a characterisation of the open
(not necessarily connected) $U\subset\hat\C$ such that
$A(U)$ is pervasive \cite[Section 4]{NOS}.
If $A(U)$ is maximal on its Shilov boundary $X=\textup{bdy}U$,
then $U$ must be connected.
It remains to be seen whether for all 
connected $U$ (with Shilov boundary
equal to $\textup{bdy}U$)
the algebra $A(U)$ is maximal.

There is reason to suppose that pervasiveness is
intimately connected with complex dimension one, and 
it has been suggested that, apart from uniform algebras
defined on closed subsets of Riemann surfaces, it might
be profitable to study two other kinds: (1) algebras
of $L^\infty$ functions on the unit circle and
other one-dimensional boundaries, particularly
subalgebras and superalgebras  
of $H^\infty(U)$, the algebra of bounded analytic
functions on a one-dimensional open set, and
(2) algebras obtained by taking the uniform closure
of the analytic polynomials on a curve lying in
the boundary of a pseudoconvex domain in $\mathbb{C}^n$.
We have something to say about type (1).  
This is connected to logmodularity.

The maximal ideal space of $H^\infty(\mathbb{D})$
is a large compact space (with cardinality greater 
than that of the continuum).  See \cite[Chapter VIII]{Garnett}
for details about its structure referred to below.
The Shilov boundary $X$
of $H^\infty(\mathbb{D})$ may be identified 
with the maximal ideal space of the self-adjoint 
uniform algebra $L^\infty(\mathbb{S}^1)$
(of essentially-bounded measurable functions
on the circle).  It is an extremally-disconnected
space, i.e. the closure of every open
set is open. The Gelfand transform of the 
identity function $z\mapsto z$ projects $M(H^\infty)$
onto the closed unit disk. The fibres over the points
of the open disk are singletons, but the preimage
of the unit circle (the {\em corona}) is
large and complicated, and $X$ is a subset of it.  
There are many interesting
uniform algebras on $X$:

\begin{enumerate}
\item[$\bullet$] $C(X)$ is just $L^\infty$, by the Stone-Weierstrass Theorem.
\item[$\bullet$] Douglas algebras, those lying between $H^\infty$ and
$L^\infty$.  We show in Subsection \ref{SS:Douglas}
that none of these is pervasive.
\item[$\bullet$] $C=C(\S^1)$ may be regarded as a subalgebra of $C(X)$,
and Sarason \cite{Sarason} showed that the vector space
sum $H^\infty+C$ is a closed subalgebra of $C(X)$. Let 
$$QC = \overline{H^\infty + C} \cap (H^\infty + C),$$ 
where the bar denotes complex conjugation.  
Then
$QC$ is sub-function algebra of $C(X)$, closed
under complex conjugation, and hence (by Stone-Weierstrass)
equal to $C(M(QC))$. The space $M(QC)$ may be regarded as
a quotient space of $X$.  Sarason introduced the algebra
$QA= QC\cap H^\infty$.   Wolff \cite{Wolff} showed
that $QA$ is a Dirichlet algebra. 
This turns out to be maximal and pervasive on $M(QC)$. 
\end{enumerate}

It might seem that the highly-disconnected nature of
$M(L^\infty)$ has a lot to do with the fact that
there are no maximal or pervasive Douglas algebras.
However,
there can be maximal algebras
on  a totally-disconnected space. The existence
of such algebras was first established by
Rudin (cf. \cite[Section 4]{HS}).

In what follows, 
in Section~\ref{background}, we present some background
results and some technical lemmas 
that will aid us throughout the paper. In Section~\ref{maximal}
we discuss some basic examples, and then  provide examples
to show that there may exist many
essentially-pervasive subalgebras of a given $C(X)$,
so that in general the gap between pervasiveness
and maximality may be very large. 
Following this, in Section~\ref{basicexamples}, 
we provide further examples of pervasive algebras as 
well as non-examples. Our examples explore the 
different possible relations between maximal 
subalgebras and pervasive subalgebras.

In Section~\ref{minimal}, we introduce some
further concepts related to support sets
for representing measures, 
and give
give some new results about pervasiveness,
maximality, and so-called relative maximality, using 
these concepts.   

We close with some questions.
\color{black}

\bigskip

\section{Preliminaries}\label{background}
\subsection{Notation and Definitions}
Throughout the paper, $A$ will denote a uniform algebra, and
$X$ its Shilov boundary. The group of invertible elements
of $A$ is denoted $A^{-1}$, and the set of continuous
functions $x\mapsto |f(x)|$
with $f\in A^{-1}$ by $|A^{-1}|$. In the same spirit,
the set of functions $\log|f|$, for 
$f\in A^{-1}$, is
denoted $\log|A^{-1}|$, and the linear vector space of real parts
$\Re f$ of the $f\in A$ is denoted $\Re A$.
For instance, we have 
 $$\Re C(X)=C(X,\R),  \, C(X)^{-1}=C(X,\C\setminus\{0\}),$$
$$|C(X)|=C(X,[0,+\infty))~\textup{and}~\log|C(X)^{-1}|=\Re C(X).$$

The algebra $A$ is said to be {\em logmodular} if $\log|A^{-1}|$
is dense in $C(X,\R)$. It is said to be {\em strongly logmodular}
if $\log|A^{-1}|=C(X,\R)$.  It is said to be {\em Dirichlet}
if $\Re A$ is dense in $C(X,\R)$.  Each Dirichlet
algebra is logmodular.

A subset $F$ of $X$ is a {\em set of antisymmetry} for $A$
if every function $f \in A$ that is real valued on $F$ is constant. 
The most important fact about antisymmetry is the 
following result of E. Bishop \cite[Theorem II.13.1, p. 60]{Ga}.
 
\begin{thm}[Bishop Antisymmetric Decomposition Theorem]  Let $A$ be a uniform algebra on $X$. Let $\{E_\alpha\}$ be the family of maximal sets of antisymmetry of $A$. Then the $E_\alpha$ are closed disjoint subsets of $X$ whose union is $X$. Each restriction algebra $A|E_\alpha$ is closed. If $f \in C(X)$ and $f|E_\alpha \in A|E_\alpha$ for all $E_\alpha$, then $f \in A$.\end{thm}

The algebra $A$ is said to be {\em antisymmetric} if 
$X$ is a set of antisymmetry. The {\em essential set} is the 
minimal closed set $E$ in $X$ such that for 
any continuous function $f$ if  $f=0$ 
on $E$, then
$f \in A$. If $E=X$, then $A$ is said to be an {\em essential algebra}. 
The restriction $A|E$ of $A$ to its essential set
is closed in $C(E)$ \cite[Theorem 2.8.1, p. 145]{Browder},
and $A$ is said to be {\em essentially pervasive} if $A|E$ is pervasive.
The algebra $A$ is said to be {\em analytic} if every function in the algebra that vanishes on a nonempty open subset of $X$ is identically zero.

\subsection{Some useful results}\label{threeusefulresults}

The following results will be used throughout the paper.  They 
summarize results that appear in \cite[pp. 220-1]{HS}. 

\begin{prop}[Hoffman-Singer] \label{HS2} Let $A$ be a uniform algebra on $X$. Then each of the following implies the subsequent statement:
\begin{enumerate}
\item $A$ is proper and pervasive.
\item $A$ is analytic on $X$.
\item $A$ is an integral domain.
\item $A$ is antisymmetric.
\item $A$ is essential on $X$.
\end{enumerate}
Moreover, every algebra that is essential and maximal is pervasive.
\end{prop}

We note that if $A$ is antisymmetric, then $M(A)$
is connected, by the Shilov idempotent theorem \cite[Cor. III.6.5, p. 88]{Ga}.

Applying this to maximal subalgebras, we get:

\begin{prop}[Hoffman-Singer] \label{HS1} Let $A$ be a maximal 
(proper, closed) subalgebra of $C(X)$. Then the 
following are equivalent:
\begin{enumerate}
\item $A$ is pervasive on $X$.
\item $A$ is analytic on $X$.
\item $A$ is an integral domain.
\item $A$ is antisymmetric on $X$.
\item $A$ is essential on $X$.
\end{enumerate}
\end{prop} 

\begin{cor} If the uniform algebra $A$ is maximal in $C(X)$,
then $A$ is essentially pervasive.
\end{cor} 
\begin{proof} If $E$ is the essential set of $A$, then
$A|E$ is maximal in $C(E)$.
Now apply the proposition to the algebra $A|E$.
\end{proof}

\subsection{Support Sets}
Given $\phi\in M(A)$, a (Borel probability) measure $\mu$ on $X$ for which 
$$\phi(f) = \int_X f \, d\mu$$ for all $f \in A$ is said to 
{\em represent $\phi$ on $A$}, or, simply, to be {\em a representing measure} 
when the context is clear.  Every $\phi$ has at least one
representing measure, and exactly one if if $A$ is logmodular \cite[Theorem II.4.2, p.38]{Ga}.

We denote the closed support of a measure $\mu$ by $\textup{supp}\, \mu$. 
For $\phi\in M(A)$, let $\mathcal{S}(\phi)$ be the family
of all sets $\textup{supp}\, \lambda$, where $\lambda$ is a measure on $X$
that represents $\phi$ on $A$. Elements of $\mathcal{S}(\phi)$ are
called {\em support sets} for $\phi$. For each $\phi$, 
there always exists at least
one minimal support set, which is either $\{\phi\}$
or is a perfect set \cite[Theorem II.2.3, p. 33]{Ga}.
For $\phi\in M(A)$, the {\em outer support of $\phi$},
denoted
$\overline{\textup{supp}\,}\phi$, is the closure
of the union
of all the support sets for $\phi$. 

We say that a measure $\lambda\in\mathcal{S}(\phi)$
{\em represents $\phi$ remotely} if $\lambda$ has no point mass 
at $\phi$.  If $\phi\in M(A)\setminus X$, then
each representing measure for $\phi$ represents it
remotely. The points of $X$ that do not have remote representing
measures are precisely the p-points, or generalised peak points
\cite[Theorem II.11.3, and a comment on p.59]{Ga}.  
We denote by $\mathcal{D}(A)$ the set of all
characters $\phi\in M(A)$ that have a remote representing
measure.
The {\em core remote support of $\phi$},
denoted by
$\underline{\textup{supp}}\, \phi$, is the intersection
of the supports of the remote representing measures for 
$\phi$. In general, this may be empty.
Evidently, $\underline{\textup{supp}}\, \phi\subseteq
\overline{\textup{supp}}\,\phi$, and if it happens that 
$\phi\in\mathcal{D}(A)$
has a unique representing measure, then its support
coincides with both the core remote and the outer supports of $\phi$.
We collect below several results that we will use in future sections.

The next proposition relies on the basic fact
(first exploited in connection with
pervasiveness by \v{C}erych \cite{Cerych})
that $A$ is dense in $C(X)$  
if and only if there exists no nonzero
annihilating measure for $A$
on $X$.

\begin{lem}\label{lemma1} Let $A$ be a uniform algebra on $X$ and let 
$\lambda$ remotely represent some $\phi\in\mathcal{D}(A)$. 
Let $E=\textup{supp} \,\lambda$.
Then $A|E$ is not dense in $C(E)$.\end{lem}

\begin{proof}  Since $E$ is not $\{\phi\}$,
we may choose $b\in E$, $b\not=\phi$. Choose $f \in~\textup{ker}(\phi)$ with $f(b)=1$.
Choose a neighborhood $N$ of $b$ such that $\Re f>\frac12$ on $N$.
Then $\Re\int_N f d\lambda>0$, so $f\lambda$ is a nonzero measure on $E$. 
But $f\lambda\perp A$, so $A|E$ is not dense in $C(E)$. 
\end{proof}

\begin{prop} Suppose $A$ is proper and pervasive on $X$, 
and $\phi\in M(A)$ is remotely represented by a measure $\lambda$ on $X$. 
Then supp\,$\lambda=X$.\label{HS?} 
\end{prop}

\begin{proof} This is immediate from the lemma.
\end{proof}

\begin{cor}\label{C:supp} Suppose $A$ is pervasive on $X$. 
Let $\phi\in\mathcal{D}(A)$. Then 
$\underline{\textup{supp}}\,\phi = X$. 
\end{cor}

\begin{cor}\label{C:perfect} If $A$ is proper and pervasive in $C(X)$, 
then $X$ is perfect.
\end{cor}
\begin{proof}
If $\phi\in M(A)\setminus X$, then it has a non-singleton 
(and hence perfect) minimal support set, 
and by Proposition~\ref{HS?}, this must be $X$.
If there is no such $\phi$, then $M(A)=X$, so $X$ is connected
(since $A$ is pervasive) and has more than one point (since $A$ is proper).
Hence $X$ has no isolated points, and so is perfect.  
\end{proof}

\section{Maximality and Pervasiveness}
\label{maximal}

\subsection{Beginning Examples}\label{SS:basex}
The disk algebra is maximal and pervasive. The algebra
of all functions continuous on the union of two
disjoint disks and holomorphic on their interiors
is neither maximal nor pervasive.

Pervasiveness says that the algebra is very big, relative to $C(X)$, yet
it is easy to give examples of maximal algebras that
are not pervasive by using Proposition \ref{HS1}.
For instance, the algebra of all functions continuous
on the union of the closed unit disk and the segment
$[1,2]$, and holomorphic on the open disk, is maximal
and not essential, hence not pervasive.  That it is maximal
is Wermer's Maximality Theorem.  Of course, this
example is essentially pervasive.

In what follows, let $[A, u_\alpha, \alpha \in I]$ denote the 
closed subalgebra of $C(X)$ generated by $A$ and the 
collection of functions $\{u_\alpha\}$.  We have the following simple proposition.

\begin{prop}\label{closetomaximal}
If $A$ is a pervasive subalgebra of $C(X)$, then $[A, \overline{f}] = C(X)$ for all nonconstant $f \in A$. \end{prop}

\begin{proof}
If $A = C(X)$, this is clear. So suppose $A$ is a proper pervasive subalgebra of $C(X)$. Let $f \in A$ be a nonconstant function and consider $B = [A, \overline{f}]$. Since $A \subseteq B$, we know that $B$ is pervasive. But $|f|^2\in B$,
so $B$ is not antisymmetric. By Proposition~\ref{HS2}, $B$ cannot
be  proper. \end{proof}

We note that this proposition does not say that $A$ is maximal. However there are situations in which $A$ must be maximal. We explore one such situation briefly, before turning to the main result in this section. We say a function $u \in A$ is {\em unimodular} if it is unimodular on $X$.

\begin{prop}

Let $A$ be a strongly logmodular proper subalgebra of $C(X)$. Then $A$ is pervasive on $X$ if and only if $A$ is essential and maximal. 

\end{prop}
\begin{proof}
First note that if $A$ is essential and maximal, then $A$ is pervasive on $X$ by Proposition~\ref{HS2}. So suppose now that $A$ is pervasive. 

Choose $f \in C(X) \setminus A$. Consider the algebra $B = [A, f]$. There exists a constant $M$ such that $f + M$ is invertible in $B$. Therefore, we may choose $g \in A^{-1}$ such that $|g^{-1}| = |f + M|$. Let $u = g(f + M)$. Then $u$ is unimodular and $u$ is invertible in $B$. Therefore $\overline{u} = u^{-1} \in B$.  Note that $u \notin A$, for otherwise we would have $f \in A$. Since $A \subset B \subseteq C(X)$, we know that $B$ is pervasive. If it were proper, it would be antisymmetric. Therefore, $B = C(X)$ and $A$ is maximal.
\end{proof}

\subsection{Non-maximal pervasive algebras}\label{SS:gem}

It is not quite
so obvious how to give an example of a non-maximal proper
pervasive algebra. De Paepe and Wiegerinck \cite{dePaepe} 
(see also \cite{HS})
gave several constructions, the simplest of which
is the algebra 
$$ \{ f\in A(\mathbb{D}): f(0)=f(1) \}.$$
We formulate a result that justifies a general
construction, and uses an elaboration of their method.

\begin{thm} Let 
$A$ be a proper pervasive algebra on $X$ containing a nonconstant 
unimodular function. Then there is an infinite descending chain
$$A \supset A_1 \supset A_2 \cdots \supset A_n \supset \cdots$$
of distinct uniform algebras, contained in $A$, each one
pervasive on $X$.
\end{thm}

\begin{proof}
Let $A$ be a proper pervasive algebra containing a nonconstant 
unimodular function $u$. Then, if $M(A) = X$, we 
would have $|u| = 1$ on $M(A)$. 
As a consequence $u$ would be invertible. 
Thus, we would have $\overline{u} = u^{-1} \in A$. 
But $A$ is antisymmetric, so this is impossible.

Therefore, since $M(A)$ is connected,
there exist distinct characters
$\phi_j \in M(A) \setminus X$ ($j\in\N$). By Proposition \ref{HS?},
each $\phi_j$ is represented by a
measure having support equal to $X$. 
So $|\phi_j(u)| < 1$, for otherwise $u$ would be constant. 
Let $a_j=\phi_j(u)$, and for $n\in\N$ let $B_n$ be the finite
Blaschke product with zeros $a_1$,$\ldots$,$a_n$.
Replacing, if need be,
$u$ by $B_n\circ u$, and taking the appropriate product we obtain $u_n \in A$ unimodular with
$\phi_j(u_n) = 0$ for each $j = 1, \ldots, n$.

Fix any $x \in X$. We know, by Corollary \ref{C:perfect},  
that $x$ is not isolated in $X$.  For each
$k\in\N$,  let 
$$A_k = \{f \in A: \phi_j(f) = \phi(x),
\hbox{ for }1\le j\le k\}.$$ 

Then each $A_k$ 
is closed and contains the constants, and the algebras $A_k$ form
an increasing chain.

For any uniform algebra $A$,
and any finite subset $F$ of $M(A)$, the restriction $A|F$
coincides with the algebra $\C^F$ of all complex-valued functions
on $F$. Thus for each $k$ with $1\le k<n$,
there is a function $f\in A$ that vanishes at $\phi_j$
for $j\le k$, but not at $\phi_{k+1}$,
so the $A_k$ are all distinct.

Now we claim that $A_n|F$ is dense in $C(F)$, whenever
$F$ is a proper closed subset of $X$. Note that this will also imply that $A_n$ separates the points of $X$ and therefore $A_n$ is a uniform algebra on $X$.

So suppose that $F = X \setminus U$ for some nonempty
open set $U$.  Let $f \in C(F)$ and let $\varepsilon > 0$. 
Suppose that $\varepsilon < 1/2$ and $\|f\| \le 1$. 

Case $1^\circ$. If $x \in F$, consider $f_1 = f - f(x)$. There exists $k \in A$ such that $\|k - f_1\|_F < \varepsilon/2$. Thus $|k(x)| < \varepsilon/2$. So $$\|(k - k(x)) - f_1\|_F \le \|k - f_1\|_F + |k(x)| < \varepsilon.$$ Note that $\overline{u_n} \in C(X)$ and therefore there exists $h \in A$ such that $\|h - \overline{u_n}\|_F = \|u_n h - 1\|_F < \varepsilon/2$. Now, $K = (k - k(x)) h u_n \in A_n$, since $K(x) = 0$ and $\phi_j(u_n) = 0$ for $j=1, \ldots, n$. Further,
$$\|K - f_1\|_F \le \|(k - k(x)) h u_n - (k - k(x))\|_F + \|(k - k(x)) - f_1\|_F.$$ But $$\|k - k(x)\|_F \le \|k\|_F + \varepsilon/2 \le \|f_1\| + \varepsilon \le 2 + \varepsilon.$$ So

$$\|K - f_1\|_F \le \|k - k(x)\|_F \cdot \varepsilon/2 + \varepsilon <  3 \varepsilon.$$
Thus, $K + f(x) \in A_n$ and $\|K + f(x) - f\|_F < 3 \varepsilon$.

Case $2^\circ$. If $x \notin F$, consider the set $F \cup \{x\}$. Then $F \cup \{x\} \ne X$, because points are not open (by Corollary \ref{C:perfect}), and $F \cup \{x\}$ is a closed set containing $x$. Thus, the previous case applies and we conclude that $A_n$ is dense in $C(F)$.

Thus $A_n$ is a pervasive algebra on $X$.
\end{proof}

If we start, for instance, with $A=A(\mathbb{D})$,
then the intersection of any infinite chain of the type
constructed in the proof might not separate 
points on $X=\S^1$, and if it did, might not be pervasive on $X$.  
One might wonder whether one could
find an infinite  descending chain with a pervasive intersection,
or an infinite ascending chain, and so on.
The next example answers all such questions.

Before leaving this theorem, we note that as one goes down
the chain of $A_j$'s in this example, the first homotopy
group $\pi_1(M(A_j))$ becomes more complex.  This might
suggest that, on a given $X$, maximal algebras 
$A$ have simplest $\pi_1(M(A))$.  However, see below.

\subsection{A Large Family of Pervasive Subalgebras}\label{longcomplicated}

We now show how to imbed the lattice of all subsets of $\N$ in
the family of pervasive subalgebras of some $C(X)$.

Our example has the additional property that all
the subalgebras are Dirichlet, 
and the least algebra in the family is generated
(as a function algebra) by one element. 
The construction depends on the following \cite[Theorem 4.1]{NOS}.

\begin{thm}\label{T:NOS} Suppose $U$ is a proper open subset of $\hat{\mathbb{C}}$ 
such that
for each boundary point $a$ of $U$ there exists some
$f\in A(U)$ with an essential singularity at $a$.
Then $A(U)$ is pervasive on the boundary $X$ of $U$ if and only if 
each connected component $U_j$ of $U$ has $X$ for its boundary.
\end{thm}

To construct the example, take $U_0$, $U_1$, $U_2$, $\ldots$
to be a countably-infinite collection of pairwise-disjoint 
simply-connected open
subsets of the sphere $\hat{\mathbb{C}}$, all sharing the same boundary $X$. 
That such a collection exists was first observed by
Brouwer \cite[p. 427]{Brouwer}. In 1917, Yoneyama \cite{Yon} gave
a nice way to describe an example of three $U_j$
that share a common boundary. His construction is known as
the \lq\lq Isles of Wada", and may be found in 
Krieger \cite[pp. 7-8]{Krieger} or on the web \cite{Web}.   
The $U_j$ are the \lq\lq ocean", a \lq\lq cold lake" and
a \lq\lq warm lake". It is easy to modify it so that there are 
infinitely-many $U_j$:
just have a separate lake with each temperature $(1/n)^\circ$C, for $n\in\N$.

Then, for each $S\subset\N$,
 let $A_S$  be the
algebra of those functions on $X$ that extend analytically across
each $U_j$ with $j\in \N\setminus S$.  It is easy to see that
$A_S=A_T$ implies $S=T$, and that 
$S\subset T$ implies $A_S\subset A_T$. Then $A_{\N}=C(X)$,
and $A_{\emptyset}$ is the intersection
of all the $A_n$, in other words, $A(\hat{\mathbb C}\sim U_0)$. 
Thus $A_\emptyset$ is Dirichlet (by the Walsh-Lebesgue Theorem \cite[Theorem II.3.3, p. 36]{Ga}),
and pervasive on $X$ (by Theorem \ref{T:NOS}), 
so that each $A_S$ is also Dirichlet and pervasive on $X$.

The maximal $A_S$'s  are those for which $\N\setminus S$
is a singleton, i.e. $A_S$ consists of all the
functions in $C(X)$ that extend holomorphically across
a single component $U_j$.  In other words, they
are the ones that have just a single nontrivial
Gleason part.  For each of these algebras, the
maximal ideal space has an infinitely-generated
first homotopy group $\pi_1(M(A))$, and the homotopy
gets {\em simpler} as we go down the lattice,
away from the maximal elements. This contrasts
with the previous example.  It suggests that
the maximal algebras might be distinguished
among the pervasive by the {\em difference}
(appropriately measured) between the topology
of $M(A)$ and that of $X$. See below.  

For general $S\subseteq\N$, the nontrivial parts of
$A_S$ are the $U_j$, for $j\in \N\setminus S$. So
a pervasive algebra may have many nontrivial parts.
In the $A_S$, the parts are all simply-connected,
but the result of Gamelin and Rossi already showed that
this is not necessary for maximality, since it applies,
for instance, to $A(U)$, where $U$ is an annulus.

\section{Examples}\label{basicexamples}

In this section, we present examples to demonstrate the relationships between various properties and pervasiveness. The disk algebra was the first example of a pervasive Dirichlet algebra. We now present some interesting Dirichlet algebras and discuss their pervasiveness (or lack thereof).

\subsection{Further Dirichlet Examples}
We give an example of an essential
Dirichlet algebra that is not pervasive.

 Recall that for a compact subset $K$ of the complex plane, $P(K)$ denotes the functions in $C(K)$ that can be uniformly approximated by polynomials in $z$ on $K$.

\begin{exam} Take three disjoint closed disks, $D_j \, \, (j=1,2,3)$,
with bounding circles $S_j$, 
and $A=P(D_1\cup D_2\cup D_3)$ on $X=S_1\cup S_2\cup S_3$,
$B=A+C(S_3)$. Then $A$ is Dirichlet and essential on $X$, $B$ is strictly between
$A$ and $C(X)$, but $B$ is not pervasive.
\end{exam}
\begin{proof}
That $A$ is Dirichlet is a case of the Walsh-Lebesgue Theorem (see \cite[p. 36]{Ga}).
Since $M(B)=X\cup D_1\cup D_2$ is not connected, $B$ is not pervasive. 
\end{proof}

The following example, of a pervasive, non-maximal
Dirichlet algebra having two nontrivial Gleason parts, 
is simpler to visualize than any of
those of Subsection \ref{longcomplicated}. 

\begin{exam} Let $X$ be a simple closed
Jordan curve which has positive area density at
each of its points.  
Then $A = A(\hat{\mathbb{C}}\sim X)$ is pervasive
on $X$, and $\mathcal D(A)$ consists of two Gleason
parts, namely the two sides of $X$. \end{exam}
\begin{proof}
The area density condition guarantees that for each
point $a\in X$, there is an element $f\in A$ having an essential
singularity at $a$, so by Theorem 4.1 of \cite{NOS}, $A$
is pervasive on $X$. 
To see that the two connected components $U_1$ and $U_2$ of 
the complement of $X$ belong to different
parts of $A$, it suffices to note that the characteristic
function $\chi_{U_1}$ of $U_1$ may be approximated, pointwise
on $U=U_1\cup U_2$, by elements of the unit ball of $A$.
This follows from work of Gamelin and Garnett:
The capacitary condition for pointwise bounded density
of $A(U)$ in $H^\infty(U)$ given in \cite{GG1}
shows that there is a bounded sequence belonging to
$A$ that approximates $\chi_{U_1}$ pointwise on $U$,
and the reduction of norm theorem \cite{GG2}
tells us that the sequence may be chosen with
sup norm bounded by $1$.   
\end{proof}

It is possible to show that $X$ may be replaced, in this
example, by any simple closed Jordan curve having no tangents, such
as the fractal snowflake.  For the essential ideas behind this
remark, see below, in Subsection \ref{SS:Warc}

\subsection{Douglas Algebras}\label{SS:Douglas} 
It is well known \cite{Garnett} that $H^\infty$ is a strongly logmodular 
subalgebra of $L^\infty$. Thus
every element of $M(H^\infty)$ has a unique representing
measure for $H^\infty$, supported on $X$.
It follows that if 
$H^\infty \subseteq B \subset L^\infty$, 
where $B$ is a closed subalgebra of $L^\infty$, 
then every element of $M(B)$ has a unique representing
measure on $X$.  We may regard $M(B)$ as a subset
of $M(H^\infty)$, by identifying
each element $\phi\in M(B)$ with the
element of $M(H^\infty)$ represented by the same
measure on $X$.   After this identification,
the Shilov boundary of $B$ is $X$.

Hoffman and Singer showed \cite[Thm 4.3, p. 222]{HS} that each 
proper pervasive algebra on a disconnected space is contained 
in a maximal algebra.  
They also showed \cite[Theorem 7.3]{HS} that $H^\infty$ is contained 
in no maximal subalgebra of $L^\infty$. Later, Sundberg \cite{Su} gave several different proofs of this fact.   Putting these facts together,
we get: 

\begin{exam} No proper closed subalgebra $B$ of 
$L^\infty$ containing $H^\infty$ can be pervasive 
on $X = M(L^\infty)$. \end{exam}

We give a direct proof that uses a little less machinery:

\begin{proof} 
Suppose that 
$H^\infty \subset B \subset L^\infty$. 
Then $\overline{z} \in B$, \cite[p. 378]{Garnett}. 
As a consequence, both $z$ and $\overline{z}$ belong to $B$, so 
$B$ cannot be antisymmetric. By Proposition~\ref{HS2}, 
$B$ cannot be pervasive on $X$.

For the case $B=H^\infty$, we may note that if $H^\infty$ were pervasive on $X$, then every Douglas algebra would be as well. 
\end{proof}

The case of $H^\infty$ can also be proved directly as follows: recall that no infinite Blaschke product is invertible in $H^\infty$. Furthermore, every infinite Blaschke product must have a zero in $M(H^\infty + C) = M(H^\infty) \setminus \mathbb{D}$. Choose a Blaschke product  $b$ with a discontinuity at $z = 1$.  Since the zeroes of $b$ cluster at $1$, we may choose $\phi \in M(H^\infty)$  with $\phi(z) = 1$ and $\phi(b) = 0$. Now $\phi$ cannot have 
$X$ as support set, for then $|\phi(z)|< 1$. If $H^\infty$ were pervasive, Proposition~\ref{HS?} would imply that $\phi \in M(L^\infty)$. But then 
$1 = |\phi (1)| = |\phi(b \overline{b})| = |\phi(b)|^2$, contradicting the fact that $\phi(b) = 0$.

\medskip
Since $X = M(L^\infty)$ is totally disconnected, one might suspect that it is impossible to have a  proper pervasive subalgebra of $C(X)$ when $X$ is totally disconnected. However, an example of this kind is
implicit in \cite[pp. 222-3]{HS}.  If $X\subset\C$ is a compact set
(such as the product $C\times C$, where $C$ is
a linear Cantor set of positive length) having positive
area in each neighbourhood of each of its points, 
and $U=\hat\C\setminus X$, then $A(U)$ is pervasive on $X$.

\subsection{Between $A(\mathbb{D})$ and $H^\infty(\mathbb{D})$}
A less well-known example of a pervasive algebra is the following.

\begin{exam} The algebra $QA$ is pervasive on $M(QC)$. \end{exam}

\begin{proof} Since $QA$ is a maximal subalgebra of QC (\cite{Wolff}) and analytic, the example follows from Proposition~\ref{HS1}. 
\end{proof}

This example has just one nontrivial part, $\mathbb{D}$. 

These results should be compared with the $CA_B$ algebras: let $B$ be a Douglas algebra properly containing $H^\infty + C$. Let $C_B$ denote the algebra generated by the unimodular functions invertible in $B$.  Then 
$B = H^\infty + C_B$, \cite{ChangMarshall2}. Furthermore, if $CA_B = C_B \cap H^\infty$, 
then the Shilov boundary of $CA_B$ is
naturally identified with $M(C_B)$, 
and if $D$ is a closed algebra with $CA_B \subset D \subset C_B$, 
then $D$ contains a nonconstant 
unimodular function invertible in $D$. In fact,
$D$ is generated over $CA_B$ by such unimodular functions. (See \cite{ChangMarshall2} for more information about these algebras.)
Therefore, $D$ is not antisymmetric and consequently no such $D$ can be pervasive. As a consequence, we may state:

\begin{lem} Let $B$ be a Douglas algebra
properly containing $H^\infty+C$. Then $CA_B$ is pervasive if
and only if it is maximal in $C_B$.
\end{lem}

We now use the lemma together with a subclass of Blaschke products to show that none of these algebras is pervasive. Recall that a Blaschke product $b$ is said to be {\it interpolating} if its zero sequence forms an interpolating sequence; that is, if $(z_n)$ is the zero sequence of $b$ and for each bounded sequence $(w_n)$ of complex numbers, there exists a bounded analytic function $f$ such that $f(z_n) = w_n$ for all $n$. In particular, the zeros of $b$ must be distinct. The Blaschke product $b$ is called {\it thin} or {\it sparse}
if it is an interpolating Blaschke product with zeros $(z_n)$ satisfying
$\lim (1 - |z_n|^2) |b^\prime(z_n)| = 1.$

\begin{exam} If $B$ is a Douglas algebra with $H^\infty + C \subset B$, then $CA_B$ is
 not maximal in $C_B$, and hence not pervasive on $M(C_B)$.\end{exam}

\begin{proof} The maximal ideal space of $H^\infty + C$ is  
$M(H^\infty) \setminus \mathbb{D}$. 
Therefore, every finite Blaschke product is 
invertible in $M(H^\infty + C)$ and no 
infinite Blaschke product is invertible in $H^\infty + C$. 
By the Chang-Marshall theorem, 
$$B = [H^\infty, \overline{b_\alpha}: b_\alpha~\textup{is a Blaschke product invertible in}~B].$$
 Since $B$ properly contains $H^\infty + C$, it must contain 
invertible infinite Blaschke products. 
Furthermore, if we factor a Blaschke product
$c = c_1 c_2$ that is invertible in $B$, then since $c_j \in H^\infty \subset B$ we see that $\overline{c_1} = c_2 (\overline{c_1 c_2}) = c_2 \overline{c} \in B$. Therefore, once a Blaschke product is invertible in the algebra, every subproduct is as well.

Choose a thin Blaschke product $b_1\in B^{-1}$, and
factorize it as $b_1 = b_{11} b_{12}$, where each factor
is an infinite Blaschke product.  There exists 
$\phi \in M(H^\infty) \setminus \mathbb{D}$ with $\phi(b_{11}) = 0$.
The support set of $\phi$ (in $M(L^\infty)$)
is a weak peak set, and therefore $H^\infty|\textup{supp}~\phi$ is closed. Now, it is known \cite{GIS} that a thin Blaschke product can have at most one zero in $M(H^\infty|\textup{supp}\, \phi)$, and $b_{11}$ already has one zero there, so $b_{12}$ cannot. Therefore, $b_{12}$ is invertible in the algebra 
$H^\infty_{\textup{supp}\, \phi}$. Thus
$|\phi(b_{12})| = 1$. Now $b_{12}$ is invertible in $B$. Let 
$D = [CA_B, \overline{b_{12}}] \subseteq C_B$. 
Now $(\phi|D) \in M(D)$ 
and $\phi(b_{11}) = 0$. Therefore, $b_{11}$ is not 
invertible in $D$ and $D \ne C_B$.
Obviously, $D\not=CA_B$, so $CA_B$ is not maximal
in $C_B$.  
\end{proof}

For algebras of functions on $\mathbb{D}$, the interested reader should consult the papers of A. Izzo, \cite{Izzo1} and \cite{Izzo2}.

\section{Maximal Algebras}
\label{minimal}

\subsection{The Extension Algebras $A_E$}
Let $A$ be a uniform algebra on its Shilov boundary $X$. 
If $B$ is a closed algebra with $A\subset B\subset C(X)$, 
then we have a map 
$$\pi:\left\{
\begin{array}{rcl}
M(B) &\to& M(A),\\
\phi&\mapsto& \phi|A.
\end{array}
\right.
$$ 
This map may or may not be surjective.
If $A$ has unique representing
measures on $X$, then $\pi$ is injective from $M(B)$ into $M(A)$. 
It may also be injective in other cases. Whenever
this happens, we identify $M(B)$ with a closed subset of $M(A)$. 
We note that in all cases, $X$ is also the Shilov 
boundary of $B$, and $\pi$ restricts to  
the identity on $X$.  However, points of $X$
may have multiple preimages, as we saw in
Subsection \ref{SS:gem}. If this happens,
then there are points of $X$ that are not
p-points for $A$.

Let $E$ be a closed subset of $X$. Then the {\it $A$-convex hull} of $E$, denoted $\hat{E}$, is the set of homomorphisms in $M(A)$ that extend continuously to the 
closure of ${A|E}$ in $C(E)$ \cite[p. 39]{Ga}. We have
$$ \hat{E} = \{\phi \in M(A): \textup{supp}\, \phi \subseteq E\}.$$  
We denote by
$$A_E :=\textup{clos}_{C(X)} \{f \in C(X): f|E \in A|E\}$$
the related function algebra on $X$. The maximal ideal space of $A_E$
is $X\cup\hat{E}$.

We observe that for each closed $E\subseteq X$,
$A$ is contained in $A_E$, that if $A$ is maximal, then $A_E$
is either maximal or is $C(X)$, and that if $A$ is 
essentially pervasive, then so is $A_E$.

\begin{lem}\label{L:A_E}
Suppose $A|E$ is closed in $C(E)$. Then \\
(1) $A_E = \{f\in C(X): f|E\in A|E\}$.\\
(2) $A_E$ is maximal in $C(X)$ if and only if
$A|E$ is maximal in $C(E)$.\\
(3) $A_E$ is essentially pervasive if and only if
$A|E$ is essentially pervasive.
\end{lem}
\begin{proof}

 (1) follows from the facts that $A|E$ is closed in $C(E)$ and $X$ is the Shilov boundary for $A$.  We present a proof of (2) and (3) for the reader's convenience.

 (2) First, suppose $A|E$ is not maximal in $C(E)$. Choose
 a uniform algebra $B$ on $E$, properly contained between them. Then 
it is easy to see that $B_E$ is properly
contained between $A_E$ and $C(X)$, so $A_E$ is not
maximal in $C(X)$.

On the other hand, suppose $A_E$ is not maximal in
$C(X)$, and choose a uniform algebra $B$ on $X$,
properly contained between them. Let 
$B_1$ be the closure in $C(E)$ of the restriction
algebra $B|E$. Then $B_1$ is a uniform algebra on
$E$, contained between $A|E$ and $C(E)$.
If $A|E=B_1$, then one checks that 
$B\subseteq B_E = A_E$, a contradiction. So 
$A|E\subset B_1$.
Suppose $B_1=C(E)$. Take a measure $\mu$ on $X$
that annihilates $B$. Then $\mu\perp A_E$,
so $\textup{supp}\,\mu\subseteq E$,
so $\mu\perp B_1$, and we conclude that $\mu=0$. Thus $B=C(X)$,
another contradiction. Thus $B_1\subset C(E)$. Thus $A|E$ is not maximal in $C(E)$.

(3) It is clear from (1) that $A_E$ and 
$A|E$ have the same essential set, and easy to
see that this set is $E\cap F$,
where $F$ is the essential set of $A$. Using (1)
again, we have $A_E|(E\cap F)= (A|E)|(E\cap F)$,
and this gives the result.
\end{proof}

\begin{lem}\label{lemma2}   
Let $\phi\in \mathcal{D}(A)$.
Then $E=\overline{\textup{supp}}\,\phi$ is a maximal antisymmetric set for 
$A_E$. 
\end{lem}

We note that $E$ need not be a maximal antisymmetric set for $A$. 
Consider, for instance,
$A(U)$, where $U$ is the union of two tangent disks in $\C$.
 
\begin{proof} Let $E=\overline{\textup{supp}}\,\phi$
and $A_1=A_E$.
The restriction homomorphism $A_1\to A_1|E$ induces
a map $M(A_1|E)\to M(A_1)$. Since $\phi(f)$
is bounded by the sup norm of $f$ on $E$,
it follows that $\phi$ belongs to $M(A_1|E)$,
and hence may be regarded as an element
of $M(A_1)$. Furthermore, it is represented on $A_1$
by each representing measure for $\phi$ on $A$.
 
If $f\in A_1$ is real-valued on $E$, then for each representing
measure $\eta_\phi$ for $\phi$ on $A$, the function
$f$ is constant on $\textup{supp}\,\eta_\phi$, with
value, say $c(\eta_\phi)$. If $\eta_\phi$ and $\xi_\phi$ both represent $\phi$
on $A$,
then $c(\eta_\phi)=\int f d\eta_\phi = \phi(f) = \int f d\xi_\phi = c(\eta_\psi)$. Thus
$f$ is constant on $E$.   Thus $E$ is an antisymmetric
set for $A_1$.
 
If $E\subset F\subseteq X$, then
there is a real-valued continuous function $f$ on $X$ that is
constant on $E$ and nonconstant on $F$. Then $f\in A_1$,
so $F$ cannot be antisymmetric for $A_1$.
\end{proof}

\begin{cor}\label{C:6}
With $E$ as in Lemma \ref{lemma2} , $A_E|E$ is closed and
antisymmetric on $E$, and hence essential. Also
$E$ is the essential set for $A_E$.\qed
\end{cor}

\begin{cor}\label{C:7} Let $\phi\in\mathcal{D}(A)$.
If $A_{\overline{\textup{supp}}\,\phi}$ is essentially pervasive, 
then 
$$\underline{\textup{supp}}\,\phi=\overline{\textup{supp}}\,\phi.$$
\end{cor}
\begin{proof} Let $E= \overline{\textup{supp}}\,\phi$.
Since $A_E|E$ is closed, Lemma \ref{L:A_E} (applied with $A=A_E$)
tells us that $A_E|E$ is essentially pervasive in $C(E)$. 
Since it is also essential,
it is pervasive on $E$. Also, 
$\phi$ belongs to $\hat E$, so $\phi\in M(A_E)\setminus X$.
Further, $\phi$ has the same set of representing
measures on $A$ and on $A_E$, so the sets
$\underline{\textup{supp}}\,\phi$ and $\overline{\textup{supp}}\,\phi$
are the same for $A$ and for $A_E$.
By Corollary \ref{C:supp}, $\underline{\textup{supp}}\,\phi=E$,
as required.
\end{proof}

\

Combining our results above, we have the following.

\begin{thm}\label{T:max-supp} Let $A$ be a uniform algebra on its 
Shilov boundary $X$. 
Suppose $A$ is essentially pervasive, and let $E$ be 
the essential set of $A$.  Let $\phi\in \mathcal{D}(A)$.
Then 
$$\underline{\textup{supp}}\,\phi=E.$$
\end{thm}

\begin{proof} Let $E'=\overline{\textup{supp}}\,\phi$.
Then by Corollary \ref{C:6}, $E'$ is the essential set
of $A_{E'}$. Since $A\subset A_{E'}$, it follows that
$E'\subset E$.
Hence $A|F$ is dense in $C(F)$ whenever $F$
is a closed proper subset of $E'$. It follows that
$A_{E'}$ is pervasive on $E'$, i.e. $A_{E'}$
is essentially pervasive. 
Applying Corollary \ref{C:7}, we get
$\underline{\textup{supp}}\,\phi=E'$. 

There is at least one representing measure
$\lambda$ for $\phi$, necessarily supported
on $E'$, and for every
$f\in$ker $ \phi$ we get an annihilating measure
$f\lambda$, supported on $E'$.  Not all
these can be zero, so $A|E'$ is not dense
in $C(E')$. Thus
$E'$ cannot be a proper subset of $E$,
so $E'=E$. The result follows. 
\end{proof}

\subsection{Other Maximal $A(U)$}\label{SS:Warc}

Consider open sets $U$ dense in
the sphere, for which $A(U)$ has nonconstant elements.
We have noted that all such $A(U)$ are pervasive.

Wermer \cite{W} considered the algebra $A(U)$ on $X$, where
$U=\hat\C\setminus X$ and $X\subset\C$ is an arc having positive
area. He showed 
that this algebra separates points on $\hat\C$, and used
it to construct an arc in $\C^3$ that is not polynomially-convex.
He and Browder \cite{BW} studied the subalgebra $A_\psi$
of $A(\mathbb{D})$
mapped isomorphically  by \lq\lq conformal welding" to $A(U)$, and 
showed that it was Dirichlet on its Shilov boundary if and
only if the associated welding map is singular. Translated
back to $U$, this says that $A(U)$ is Dirichlet if and
only if the two pieces of harmonic measure 
(say for the point $\infty$)
on the two sides of $X$ are mutually-singular. Their
proof used the F. and M. Riesz Theorem,
and an analysis of $A_\psi$ as a subalgebra
of the disk algebra.  Their
result applies to any arc $X$, and characterises those
for which $A(U)$ is Dirichlet on $X$.  They
managed to construct an example by applying
the Ahlfors-Beurling quasiconformal
extension theorem and quasiconformal welding.

Later, Bishop, Carleson, Garnett and Jones \cite{BCGJ}
showed that the two pieces of harmonic measure are
mutually-singular as soon as $X$ has no tangents.
Thus $A(U)$ is Dirichlet whenever $X$ has no tangents.

\begin{exam} If $X\subset\C$ is an arc having no tangents,
and $U=\hat\C\setminus X$, then $A(U)$
is maximal in $C(X)$.
\end{exam}

For the reader's convenience, we give a direct 
proof that establishes this
without transferring to the unit disk.  
A proof of this kind has not appeared in print,
but see \cite[Exercise 1(g), p. 63]{Ga}.  The key
point used here is that $A(U)$ is Dirichlet.  The
Dirichlicity can also be established without
passing to the disk, by applying the capacity
condition of Gamelin and Garnett, and a result 
of Bishop.

\begin{proof}
By Arens' Theorem \cite[Theorem II.1.9, p.31]{Ga}, 
the maximal ideal space
of $A(U)$ is $\hat\C$.
Since $A=A(U)$ is Dirichlet, it has unique representing 
measures, so its harmonic measure is the only representing
measure for a point $a\in U$ on $A$. Suppose
$B$ is a uniform algebra contained between 
$A$ and $C(X)$, and consider the map
$\pi:M(B)\to M(A)$. 

If there is some point $a\in U$ that is omitted by $\pi$,
then $z\mapsto 1/(z-a)$ belongs to $B$,
and thus $B$ contains the closure on $X$ of the algebra
of functions holomorphic near $X$, hence $B=C(X)$.
 
So suppose the image of $\pi$ contains $U$. Then
$\pi$ is injective, since for $a\in U$
all points of $\pi^{-1}(a)$
must share the same representing measure,
namely the harmonic measure for $a$.
Thus $B$ is an algebra of functions on
$\hat\C$. Moreover, since it is represented
on $U$ by the harmonic measures, 
each $f\in B$ is harmonic on $U$.
Thus $gf$ is harmonic on $U$,
for each $g\in A$ and $f\in B$.
But then
$$ 0=\Delta(fg)=\partial\bar{\partial}(fg)
= \partial(g\bar{\partial}f) =
(\partial g)( \bar{\partial}f)
$$
on $U$.
It is not hard to see that for each $a\in U$
there is some $g\in A$ with $\partial g(a)\not=0$
(just take a nonconstant $g\in A$, and if
$g'(a)=0$, consider the first $k\in\N$
with $g^{(k+1)}(a)\not=0$, and form
$(g(z)-g(a))/(z-a)^k$),
so we conclude that $f$ is holomorphic on $U$.
Thus $B=A$.
\end{proof}

We remark that Hoffman and Singer constructed
a maximal algebra on an arc, by starting with
a similar $A(U)$, but having complement
consisting of two arcs, and
forming a quotient algebra.  They left open
the question whether such an $A(U)$
could be maximal.

\subsection{Relative Maximality}

In this section we study relative maximality, a concept that permits strongly logmodular algebras to be maximal relative to an algebra other than $C(X)$. Throughout this subsection, $A$ will be a logmodular subalgebra of $C(X)$ and $B$ will represent
a closed algebra with $A\subset B\subseteq C(X)$.

Given a set $\Omega \subset M(A)$, Guillory and Izuchi study relative minimal support sets in the context of subalgebras of $L^\infty$ containing $H^\infty$, (see \cite{GI}, as well as \cite{GIS}). Our result is motivated by their work. Thus, we will say a point $\phi \in M(A) \setminus M(B)$ is a {\it minimal support point
(relative to $M(A) \setminus M(B)$)} if there is no point $\psi \in M(A) \setminus M(B)$ with $\textup{supp}\,\psi \subset \textup{supp}\,\phi$. The goal of this section is to prove Theorem~\ref{minimpliesmax}, below, which applies to Douglas algebras. Douglas-like subalgebras of $L^\infty(\mathbb{D}, dA)$ (where $dA$ denotes area measure on $\mathbb{D}$) have recently been investigated \cite{AxlerZheng}.

Before we consider relative maximal algebras, we should note that if $\textup{supp}\, \phi$ is a proper nontrivial subset of $X$, then the algebra $A_{\textup{supp}\,\phi}$ is not essential and therefore not pervasive.

\begin{thm}\label{minimpliesmax}  Let $A$ be a logmodular algebra on  $X$ and
 $B$ an algebra with $A \subset B \subseteq C(X)$. Suppose that

\begin{enumerate}
\item[a)]  for each algebra $D$ with $D \subseteq C(X)$ there exists a nonempty index set $I$ and unimodular functions 
$$\mathcal{U} = \{u_\alpha: \alpha \in I\} \subset A ~\mbox{such that}~ D = [A, \overline{u_\alpha}: u_\alpha \in \mathcal{U}] ~\mbox{and}~$$
\item[b)] whenever $B_1, B_2$ are subalgebras of $C(X)$ containing $A$ and $\psi \in M(A) \setminus \big(M(B_1) \cup M(B_2)\big)$, there exists $u \in A$, unimodular, such that $0 \notin u(M(B_1) \cup M(B_2))$ but $\psi(u) = 0$.
\end{enumerate} Then $\phi \in M(A) \setminus M(B)$ is a minimal support point
if and only if $B \cap A_{\textup{supp}\, \phi}$ is maximal in $B$. 
\end{thm}

The proof depends on the following lemma.

\begin{lem} Let $A$ be a logmodular algebra. Suppose that whenever $B_1$ and $B_2$ are subalgebras of $C(X)$ with $A \subset B_1 \cup B_2$ and $\psi \in M(A) \setminus \left(M(B_1) \cup M(B_2)\right)$, then there exists $u \in A$, unimodular, such that $0 \notin u(M(B_1) \cup M(B_2))$ but $\psi(u) = 0$. Then
\begin{eqnarray}\label{maximalideal} 
M(B \cap A_{\textup{supp}\, \phi})  = M(B)  \cup M(A_{\textup{supp}\, \phi}). \end{eqnarray}

 \end{lem}

\begin{proof}

We let $A_1 = B \cap A_{\textup{supp}\, \phi}$. 

One containment is obvious, so suppose $\psi \in M(B \cap A_{\textup{supp}\, \phi}) = M(A_1).$ In particular, $\psi \in M(A)$.  Suppose that $\psi \notin M(B) \cup M(A_{\textup{supp}\, \phi})$.  By our assumption, there exists $u \in A$ such that $u$ does not vanish on $M(A_{\textup{supp}\, \phi}) \cup M(B)$, but $\psi(u) = 0$. Therefore, $\overline{u} \in A_1$. But $\psi \in M(A_1)$ and $u \in A_1^{-1}$, so this is impossible.

\end{proof}

\begin{proof}[Proof of Theorem~\ref{minimpliesmax}]
Assume first that $\textup{supp}\, \phi$ is minimal. If $A_1$ is not maximal, there exists an algebra $D$ with $A_1 \subset D \subset B$. By assumption (part (a)), there exists $u \in A \cap D^{-1}$ unimodular and $u \notin A_1^{-1}$. Since $u \in D^{-1}$, we know that $\overline{u} \in B$. Thus, $u \notin A_1^{-1}$ implies that $\overline{u} \notin A_{\textup{supp}\, \phi}$. Thus, $|\phi(u)| < 1$ and $u$ is not constant on the support of $\phi$. 

Fix $\tau \in M(D)$ and note that $|\tau(u)| = 1$. If $\tau \notin M(B)$, then $\tau \in M(A) \setminus M(B)$. Since $\tau \in M(A_1)$, we know from the lemma above that, in this case, $\tau \in M(A_{\textup{supp}\, \phi})$. Since the support of $\phi$ is minimal, we have $\textup{supp}\, \phi = \textup{supp}\, \tau$. But $u$ not constant on the first set, while it is constant on the latter set. Therefore, this is impossible. Therefore, $\tau \in M(B)$. Thus $M(D) = M(B)$ and every unimodular function in $A$ that is invertible in $B$ is also invertible in $D$. Thus $B = D$, and $A_1$ is maximal. 

Now suppose that $A_1$ is maximal. Suppose that $\psi \in M(A) \setminus M(B)$ and $\textup{supp}\, \psi \subset \textup{supp}\, \phi$. Then 
$\psi \in M(A_{\textup{supp}\, \phi}) \subseteq M(A_1)$. By hypothesis (part (b)), however, there exists $u \in A$ unimodular such that $u \in B^{-1}$ and $\psi(u) = 0$.  Now  $\phi \notin M(B) \cup M(A_{\textup{supp}\, \psi})$. Letting $A_2 = B \cap A_{\textup{supp}\, \psi}$ and again applying hypothesis (b), we find that there exists $v \in A_2^{-1}$, unimodular, such that $\phi(v) = 0$.  Therefore $A_1 \subset A_2$, and since $\overline{u} \in B \setminus A_2$, we see that $A_1$ is not maximal, a contradiction. \end{proof}

The next example illustrates the relative maximality discussed above and can be found in \cite{GI}, but we present the short proof here. Note that every closed  subalgebra $B$ of $L^\infty$ with $B \supset H^\infty$ is, by the Chang-Marshall theorem, generated by $H^\infty$ and the conjugates of interpolating Blaschke products. Furthermore, the second hypothesis of Theorem~\ref{minimpliesmax} is satisfied as well, though the justification is more work: Suppose that $H^\infty \subset B_j$ for $j = 1, 2$ and $\psi \in M(H^\infty) \setminus (M(B_1) \cup M(B_2))$. Choose an open set $V$ about $\psi$ that is disjoint from $M(B_1) \cup M(B_2)$. By \cite[Corollary 3.2]{GorkinMortini}, there is a point $\psi_0$ in $V \cap M(H^\infty_{\textup{supp}\, \psi})$ such that $\psi_0$ is in the closure of an interpolating sequence $(z_n)$.  Choose a subset of $(z_n)$, denoted by $(z_{n, 1})$, capturing $\psi_0$ in its closure and such that the closure of the sequence is entirely contained in $V$. Then a result of Hoffman \cite{H1} implies that the corresponding Blaschke product $b$ can vanish at $x$ in $M(B_j)$ if  and only if $x$ is in the closure of $(z_{n, 1})$. Therefore, $b$ will not vanish on $M(B_j)$, while $\psi_0(b) = 0$. Since the support set of $\psi_0$ is contained in that of $\psi$, we see that $|\psi(b)| < 1$. Then $u = (b - \psi(b))/(1 - \overline{\psi(b)} b)$ satisfies the hypothesis of Theorem~\ref{minimpliesmax}, part (b).

\bigskip
\begin{exam} \cite{GI} The following is an example of a minimal support set in a strongly logmodular algebra on a totally disconnected space. 
\end{exam}

This result relies on the fact that any point in the closure of a thin sequence has a maximal support set. This is an unpublished result of Hoffman and can be found in \cite{Budde} .

\begin{proof} Let $b$ be a thin Blaschke product; that is, the zeroes of $b$ are $(z_n)$ and they satisfy
$\lim (1 - |z_n|^2) |b^\prime(z_n)| = 1.$  Consider $H^\infty[\overline{b}]$ and let $\phi \in M(H^\infty) \setminus M(H^\infty[\overline{b}])$. We claim that $\phi$ has a minimal support set; that is, if $\psi \in M(H^\infty) \setminus M(H^\infty[\overline{b}])$ then the support of $\psi$ is not properly contained in the support of $\phi$. To see this, we argue by contradiction: Suppose that $\textup{supp}\, \psi \subset \textup{supp}\, \phi$.  Since $\psi \in M(H^\infty) \setminus M(H^\infty[\overline{b}])$, we must have $|\psi(b)|< 1$. So $\overline{b} \notin H^\infty|\textup{supp}\, \psi$ and there exists $\psi_1 \in M(H^\infty|\textup{supp}\,\psi)$ with $\psi_1(b) = 0$. Therefore,  (by Hoffman's result, as stated in \cite{Budde}) $\textup{supp}\, \psi_1$ is maximal, so $\textup{supp}\,\psi_1 = \textup{supp}~\phi$. But then $\textup{supp}\, \psi = \textup{supp}~\phi$ as well, establishing the contradiction. So $\phi$ is a minimal support point and $H^\infty_{\textup{supp}\,\phi} \cap H^\infty[\overline{b}]$ is maximal in $H^\infty[\overline{b}]$. 

\end{proof}

This algebra isn't pervasive, of course, because it is a subalgebra of $L^\infty$ containing $H^\infty$.

\bigskip
\begin{exam} Let $A = H^\infty$ and $B = H^\infty + C$. Then $H^\infty$ is maximal in $H^\infty + C$ and every point in $M(A)\setminus M(B)$ has minimal support set. Further, $H^\infty$ is relatively pervasive in $H^\infty + C$ (in a sense to be made precise below). \end{exam}

\begin{proof} It is well known that $H^\infty$ is maximal in $H^\infty + C$ (\cite[p. 376]{Garnett}) and that $M(A) \setminus M(B) = \mathbb{D}$; that is, evaluation at points of the open unit disk $\mathbb{D}$. 

The algebra $H^\infty$ is also ``relatively'' pervasive in $H^\infty + C$ in the following sense: Let $F$ be a closed and proper subset of the Shilov boundary $X$. Then $U = X \setminus F$ is open. Choose a smaller open set $U_1 \subset U$ such that its closure is also contained in $U$. Since $X$ is extremally disconnected, $W_1 = \textup{clos}(U_1)$ is clopen. Thus, by (\cite[p. 376]{Garnett}) the closed subalgebra of $L^\infty$ generated by $H^\infty$ and $\chi_{W_1}$ contains $H^\infty + C$. So for each $h \in H^\infty + C$, there exist $f_n, g_n \in H^\infty$ such that 
$\|h - (f_n + g_n \chi_{W_1})\|_X \to 0.$ Since $\chi_{W_1} = 0$ on $F$, we have $\|h - f_n\|_F \to 0,$ completing the proof.
\end{proof}

\section{Final Questions}
Our results and examples raise some questions:

1. We have seen that a maximal algebra $A$ 
on $X$ may have $M(A)$
either more or less complicated, topologically,
than $X$. We have also seen that a maximal
algebra need not have a homotopically trivial
$M(A)$. What, if any, topological conditions
on $X$ and $M(B)$ will guarantee that an
essentially pervasive $B$ is maximal?  
It can't just be a matter of comparing
$\pi_1$'s, because of the example
in Subsection \ref{SS:Warc}.

2. One common feature of all our examples
is that the maximal algebras have just one
nontrivial part, but this is not necessary
for maximality. A theorem of
Hoffman and Singer  \cite[Theorem 4.7]{HS2}
implies that
the big disk algebra is maximal on the torus,
and this algebra has many nontrivial parts.
However, each nontrivial part of the
big disk algebra is dense in the whole
$M(A)$, so one may ask whether it is necessary
for maximality
that $M(A)$ lies in the closure
of each nontrivial part.  
However, even if necessary for maximality,
this property will not be enough
to distinguish the maximal from the 
essentially pervasive. The counterexamples
we have seen 
have the feature that the part meets $X$.

3. Suppose $A$ has just one nontrivial part $P$, and $P$ is finitely-connected,
and the weak-star and metric topologies agree on $P$. Is $A$
maximal in $C(X)$?  What if we assume that all points
of $X$ are peak points?

4. Can a pervasive 
algebra have more than one non-simply-connected nontrivial part?
What about a maximal algebra?

5. Can a strongly logmodular essential algebra be pervasive? 

6. It would
be interesting if we could even show that you can't have
a proper pervasive algebra on $X=M(A)=[0,1]$. As far as we know,
this is open.

7. Is $A(U)$ maximal in $C(X)$, whenever $U\subset\hat{\mathbb{C}}$
is open and connected, and
$X=\textup{bdy}(U)$ is the Shilov boundary of 
$A(U)$?

8. Suppose $A$ is essential and $\phi\in \mathcal{D}(A)$ implies
$\underline{\textup{supp}}\,\phi=E$ and there are no
completely-singular annihilating measures for $A$.
Is $A$ pervasive?

\section*{Acknowledgments}
The second author was partially-supported by
the HCAA network.

Part of this work was done at the meeting
on Banach Algebras held at 
Bedlewo in July 2009. The support
for the meeting by the Polish Academy of Sciences, the European Science Foundation under
the ESF-EMS-ERCOM partnership, and the Faculty of Mathematics and Computer Science
of the Adam Mickiewicz University at Poznan is gratefully acknowledged. The first author is grateful to the London Mathematical Society for travel funding.

\end{document}